\documentclass[11pt,a4paper]{article}
\usepackage[margin=25mm]{geometry}
\usepackage[pdftex]{graphicx}

\usepackage[pdftex,usenames,dvipsnames]{xcolor}
\usepackage[utf8]{inputenc} 
\usepackage{amsfonts,amsmath,amssymb}
\usepackage{theorem}
\usepackage[english]{babel}

\def\e{\varepsilon}
\def\e{\varepsilon}
\def\ol#1{\overline{#1}}

\def\R#1{\ensuremath{\mathbb{R}^{#1}}}
\def\S#1{\ensuremath{\mathbb{S}^{#1}}}

\def\<{\left<}
\def\>{\right>}
\def\e{\varepsilon}
\def\g{\gamma}

\newenvironment{proof}{\ignorespaces\par\noindent\textbf{Proof.}\quad}{\hfill$\Box$\par}

\def\CRPC{C\kern-.5pt R\kern-.5pt P\kern-.5pt C}
\def\CSkC{CSkC}

{\theorembodyfont{\itshape}
\newtheorem{theo}{Theorem}
\newtheorem{prop}[theo]{Proposition}

} {\theorembodyfont{\upshape}

\newtheorem{defi}{Definition}
\newtheorem{example}{Example}
}

\title{A property that characterizes the Enneper surface and helix surfaces}
\author{\normalsize Pascual Lucas\footnote{Corresponding author.\newline \hspace*{17pt}E-mail
addresses: plucas@um.es and yagues1974@hotmail.com}\and\normalsize
José Antonio Ortega-Yagües}
\date{\normalsize Departamento de Matemáticas, Universidad de Murcia\\
       Campus de Espinardo, 30100 Murcia SPAIN\\[10mm] \today}

\begin{document}

\maketitle

\begin{abstract}
The main goal of this paper is to show that helix surfaces and the Enneper surface are the only surfaces in the 3-dimensional Euclidean space $\R3$ whose isogonal lines are generalized helices and pseudo-geodesic lines.
\end{abstract}
\medskip



\section{Introduction}
\label{s:intro}

Let $\gamma(s)\subseteq\R3$, $s\in I$, be an arc-length parametrized curve, with Frenet frame $\{T_\g,N_\g,B_\g\}$, satisfying the Frenet equations
\begin{equation}\label{FS-eq}
T_\gamma' = \kappa_\gamma\,N_\gamma,\qquad
N_\gamma'= -\kappa_\gamma\,T_\gamma-\tau_\gamma\,B_\gamma,\qquad
B_\gamma' = \tau_\gamma\,N_\gamma,
\end{equation}
where, as usual, $\kappa_\g>0$ and $\tau_\g$ stand for the \emph{curvature} and \emph{torsion} functions of $\g$.

When the curve $\gamma$ lies on an orientable connected surface $M\subseteq\R3$, with Gauss map $N$, then along $\gamma$ we also have the Darboux frame $\{T_\gamma,JT_\gamma,N\}$, where $J$ denotes the positively oriented $\pi/2$-rotation in the tangent plane, which is defined by $JV=N\times V$. In this case, we have the following equations 
\begin{equation}\label{skntg}
T'_\gamma= \kappa_g\,(JT_\gamma)+\kappa_n\,N,\qquad
(JT_\gamma)' = -\kappa_g\,T_\gamma -\tau_g\,N,\qquad
N' = -\kappa_n\,T_\gamma +\tau_g\,(JT_\gamma),
\end{equation}
where the functions $\kappa_g$, $\kappa_n$ and $\tau_g$ are called the \emph{geodesic curvature}, \emph{normal curvature} and \emph{geodesic torsion} of $\g$, respectively. From first equations of (\ref{FS-eq}) and (\ref{skntg}), we easily get the well-known relation $\kappa_\gamma^2 =\kappa^2_g+\kappa^2_n$.

Let $\{E_1,E_2\}$ be a local orthonormal frame of principal directions of $M$, which is positively oriented, associated to the principal curvatures $\kappa_1$ and $\kappa_2$, with $\kappa_1\leq\kappa_2$. The orthonormal frames  $\{T_\g,N_\g,B_\g\}$ and $\{E_1,E_2,N\}$ in $\R3$ are related along $\gamma$ by
\begin{align}\label{refdarboux}
 T_\gamma= & \cos\phi\,E_1+\sin\phi\,E_2,\nonumber\\
 N_\gamma= & \sin\theta\,(-\sin\phi\,E_1+\cos\phi\,E_2)+\cos\theta N,\\
 B_\gamma= & -\cos\theta\,(-\sin\phi\,E_1+\cos\phi\,E_2)+\sin\theta N,\nonumber
\end{align}
for certain differentiable angle functions $\phi$ and $\theta$. From the equations (\ref{FS-eq})--(\ref{refdarboux}) it is not difficult to see that
\begin{equation}\label{kgkn}
\kappa_g =\sin\theta\,\kappa_\gamma,\quad\kappa_n =\cos\theta\,\kappa_\gamma.
\end{equation}
The following relations are also well-known (see \cite[pp. 145, 153]{DoCarmo76}):
\begin{equation}\label{kntg}
\kappa_n=\kappa_1\cos^2\phi+\kappa_2\sin^2\phi,\qquad
\tau_g=(\kappa_1-\kappa_2)\cos\phi\sin\phi.
\end{equation}

We recall the following definition introduced by Santaló in \cite{Santalo43}.
\begin{defi}
A curve $\gamma$ in a surface $M$ is said to be an \emph{isogonal line in $M$} if the angle $\phi$ between $T_\g$ and $E_1$ is constant along $\g$.
\end{defi}
The lines of curvature and the loxodromes are trivial examples of isogonal lines. Note that any reparametrization of an isogonal line is also an isogonal line with the same angle. Therefore, and without loss of generality, we can assume that our isogonal lines have constant velocity.

Geodesics $\g$ on a surface $M$ are characterized by having their principal normal vector field $N_\gamma$ parallel to the unit vector field $N$ normal to the surface. This property was extended in a very natural way by Wunderlich in \cite{W49a,W49b,W50,W51}.
He introduced the concept of pseudo-geodesic lines on a surface as the curves traced on the surface whose osculating planes form a fixed angle with the tangent plane of the surface; when this angle is a right angle, we get the proper geodesics and when it is equal to zero, the asymptotic lines. This property is equivalent to the following.
\begin{defi}
A curve $\g$ in a surface $M$ is said to be a \emph{pseudo-geodesic line in $M$} if the angle $\theta$ between the normals $N_\gamma$ and $N$ is constant along $\gamma$.
\end{defi}
It is easy to see that every curve in a plane is a pseudo-geodesic line in that plane, and for curves $\g$ in the 2-sphere $\S2(r)$, $\g$ is a pseudo-geodesic line if and only it is a circle.
A link between pseudo-geodesic lines and helices has recently been found in \cite{LO23}, where the generalized helix concept has been extended. It is shown in \cite{LO23} that normal helices are precisely the pseudo-geodesic lines in a generalized cylinder.

Along this paper we will use the usual concept of generalized helix.
\begin{defi}
A curve $\gamma$ is a \emph{generalized helix}, or \emph{cylindrical helix}, if its tangent vector makes a constant angle with a fixed direction $v$, called an axis of the generalized helix.
\end{defi}
It is well-known (see, for example, \cite[p. 20]{Eis09} or \cite[p. 26]{DoCarmo76}) that $\gamma$ is a generalized helix if and only if $m\kappa_\gamma+n\tau_\gamma=0$, for certain constants $m$ and $n$.

This paper is organized as follows. In Sect. \ref{s:CCS} we obtain some results on isogonal lines, and their relation to several special families of surfaces (\emph{\CRPC--surfaces} and \emph{\CSkC--surfaces}). In Sect. \ref{s:CCNS} we study pseudo-geodesic lines and obtain the system of differential equations that characterizes them, which obviously is very similar to the system that characterizes the geodesic curves. In that section we also obtain a result analogous to the well-known Joachimsthal theorem (see Proposition \ref{JoachimsthalEx2}). Sect. \ref{s:MR} is devoted to prove the main result of the paper (see Theorem \ref{MR}): \emph{Let $M$ be a nonplanar connected surface in $\R3$. Then the isogonal lines are pseudo-geodesic lines and generalized helices if and only if $M$ is a helix surface or an open piece of the Enneper surface}. We end the paper by providing a family of curves on the Enneper surface with nice properties.

The authors would like to thank the reviewer for his comments and suggestions that have improved the original manuscript.

\section{Some results on isogonal lines}
\label{s:CCS}\label{s:NIL}

Our first goal in this section is to prove that given a point $p$ in a surface $M$ and a tangent vector $v\in T_pM$, then there exists a unique isogonal line with initial conditions $(p,v)$.

Let $X(t,z)$ be an orthogonal parametrization in a surface $M$, and consider $\g(s)=X(t(s),z(s))$ an isogonal line with angle $\phi$ satisfying $\g(0)=p$ and $\g'(0)=v$. Then its tangent vector is given by
\begin{equation}\label{eq1}
\g'(s)=t'(s)X_t(t(s),z(s))+z'(s)X_z(t(s),z(s)).
\end{equation}
Let us write
\begin{equation}\label{eq2}
X_t(t,z)=f_1(t,z)E_1(t,z)+f_2(t,z)E_2(t,z),\qquad
X_z(t,z)=g_1(t,z)E_1(t,z)+g_2(t,z)E_2(t,z),
\end{equation}
where $f_1,f_2,g_1,g_2$ are differentiable functions. Equations (\ref{eq1}) and (\ref{eq2}) leads to
\begin{equation}\label{sedo-cpc}
t'(s)f_1(s)+z'(s)g_1(s)=|v|\cos\phi,\qquad t'(s)f_2(s)+z'(s)g_2(s)=|v|\sin\phi,
\end{equation}
where we have written, for simplicity, $f_1(s)=f_1(t(s),z(s))$, and so on. It is not difficult to see that the differential equations (\ref{sedo-cpc}) characterize the  isogonal lines. So we have proved the following result.

\begin{prop}\label{p1.a}
Given a point $p\in M$ and a vector $v\in T_pM$, $v\neq0$, there exist an $\e>0$ and a unique  isogonal line $\gamma:(-\e,\e)\to M$ such that $\g(0)=p$, $\g'(0)=v$ and $|\g'(t)|=|v|$ for all $t\in(-\e,\e)$.
\end{prop}
To indicate the dependence of this isogonal line of $(p,v)$, it is convenient to denote it by $\g(t,p,v)$.

\begin{prop}\label{p1.b}
If the  isogonal line $\g(t,p,v)$ is defined for $t\in(-\e,\e)$, then the isogonal line $\g(t,p,\lambda v)$, $\lambda\in\R{}$, $\lambda\neq0$, is defined for $t\in(-\e/\lambda,\e/\lambda)$, and $\g(t,p,\lambda v)=\g(\lambda t,p,v)$.
\end{prop}
\begin{proof}
Let $\alpha:(-\e/\lambda,\e/\lambda)\to M$ be the parametrized curve defined by $\alpha(t)=\g(\lambda t,p,v)$. Then $\alpha(t)$ is an isogonal line such that $\alpha(0)=p$, $\alpha'(0)=\lambda v$ and $|\alpha'(t)|=|\lambda v|$ for all $t$, and so from Proposition \ref{p1.a} we have $\alpha(t)=\g(t,p,\lambda v)$.
\end{proof}

If $v\in T_pM$, $v\neq0$, is such that $\g(1,p,v)$ is defined, we set $\Phi_p(v)=\g(1,p,v)$ and $\Phi_p(0)=p$. Then we can define a differentiable map $\Phi_p$ in some neighborhood of $0\in T_pM$. This map will be called the \emph{isogonal map} at $p$. More precisely,

\begin{prop}\label{p1.c}
Given a point $p\in M$ there exists an $\e>0$ such that $\Phi_p$ is defined and differentiable in the interior $B_\e$ of a disk of $T_pM$, with center in the origin and of radius $\e$.
\end{prop}
\begin{proof}
It is clear that for every direction of $T_pM$ it is possible, by Proposition \ref{p1.b}, to take $v$ sufficiently small so that the interval of definition of $\g(t,p,v)$ contains 1, and thus $\g(1,p,v)=\Phi_p(v)$ is defined. To show that this reduction can be made uniformly in all directions, we need the theorem of the dependence of an  isogonal line on its initial conditions, in the following form: \emph{Given $p\in M$ there exist numbers $\e_1>0$, $\e_2>0$ and a differentiable map $\g:(-\e_2,\e_2)\times B_{\e_1}\to M$ such that, for every $v\in B_{\e_1}\subset T_pM$, $v\neq0$, $t\in(-\e_2,\e_2)$, the curve $\alpha(t)=\g(t,p,v)$ is the  isogonal line in $M$ with $\alpha(0)=p$, $\alpha'(0)=v$ and $|\alpha'(t)|=|v|$; and for $v=0$, $\g(t,p,0)=p$}.\\
From this statement and Proposition \ref{p1.b}, our assertion follows. In fact, since $\g(t,p,v)$ is defined for $|t|<\e_2$, $|v|<\e_1$, we obtain, by setting $\lambda=\e_2/2$ in Proposition \ref{p1.b}, that $\g(t,p,(\e_2/2)v)$ is defined for $|t|<2$, $|v|<\e_1$. Therefore, by taking a disk $B_{\e_1}\subset T_pM$, with center at the origin and radius $\e<\e_1\e_2/2$, we have that $\g(1,p,v)=\Phi_p(v)$, $v\in B_{\e_1}$, is defined. The differentiability of $\Phi_p$ in $B_{\e_1}$ follows from the differentiability of $\g$.
\end{proof}
The differentiable map $\g(t,p,v)$ will be called the \emph{isogonal flow}.

\begin{prop}\label{p1.d}
Given a point $p\in M$, the map $\Phi_p:B_\e\subset T_pM\to M$ is a diffeomorphism in a neighborhood $U\subset B_\e$ of the origin $0$ of $T_pM$.
\end{prop}
\begin{proof}
We shall show that the differential $d(\Phi_p)_0$ is nonsingular at $0\in T_pM$. To do this, we identify as usual the space of tangent vectors to $T_pM$ at $0$ with $T_pM$ itself. Then
\[
d(\Phi_p)_0(v)=\left.\frac{d}{dt}\right|_{t=0}\Phi_p(tv)=\left.\frac{d}{dt}\right|_{t=0}\g(t,p,v)=v.
\]
It follows that $d(\Phi_p)_0$ is the identity map. By applying the inverse function theorem, we complete the proof of the proposition.
\end{proof}
\medskip

A surface $M$ with \emph{constant ratio of principal curvatures} or \emph{\CRPC--surface} is a surface such that $a\kappa_1+b\kappa_2=0$ for certain constants $a,b\in\R{}$ with $a^2+b^2\neq0$; if we assume $\kappa_1\neq0$, then a \emph{\CRPC--surface} is a surface with constant $\kappa_2/\kappa_1$, \cite{JMP20,WP22}. In other words, a \emph{\CRPC--surface} is a surface with linearly dependent functions $\kappa_1$ and $\kappa_2$. Trivial examples are spheres and minimal surfaces. Another known case is an ideal \emph{Mylar balloon} (see, e.g., \cite{MO03,MO07}) which is obtained by gluing two equally sized discs of ﬂexible, but inextensible, foil along their common border, and blowing it up. Further surfaces of revolution of that sort, with positive and negative ratios of principal curvatures, have appeared in other contexts, e.g., in \cite{Hopf, Kuhnel06, JMP20}.

\begin{prop}\label{p1.1}%
Let $\gamma$ be an arclength parametrized curve in a surface $M$. If $\gamma$ is an isogonal line (which is not a line of curvature), then the following conditions are equivalent:\vspace*{-\topsep}
\begin{enumerate}\itemsep0pt\def\labelenumi{\roman{enumi})}
\item[C1)]
$\kappa_1$ and $\kappa_2$ are linearly dependent functions along $\gamma$.
\item[C2)]
$\tau_g$ and $\kappa_n$ are linearly dependent functions along $\gamma$.
\end{enumerate}
\end{prop}
\begin{proof}
From equations (\ref{kntg}), it is easy to see that, for any constants $a,b,c,d$ and $\phi$, we have
\begin{align}
\sin\phi\cos\phi(a+b)\kappa_n+(a\sin^2\phi-b\cos^2\phi)\tau_g &= \sin\phi\cos\phi(a\kappa_1+b\kappa_2).\label{(1)}\\
(d\cos\phi+c\sin\phi)\cos\phi\kappa_1+(d\sin\phi-c\cos\phi)\sin\phi\kappa_2 &= c\tau_g+d\kappa_n.\label{(2)}
\end{align}
From these equations, it is not difficult to see that \emph{C1} and \emph{C2} are equivalent conditions.
\end{proof}

\begin{prop}\label{p1.2}
Let $\gamma$ be an arc-length parametrized curve in a surface $M$, $\g$ without umbilical points. If $\g$ satisfies conditions C1 and C2 of Proposition \ref{p1.1}, then $\g$ is an isogonal line.
\end{prop}
\begin{proof}
It is an easy consequence of the equation (\ref{(2)}).
\end{proof}

A surface $M$ is said to be of \emph{constant skew curvature} or \emph{\CSkC--surface} if the difference of principal curvatures $\kappa_1-\kappa_2=\lambda$ is a constant, \cite{TP17}. Note that this is equivalent with $H^2-K=\lambda^2/4$ being constant, where $H$ is the mean curvature and $K$ is the Gaussian curvature of the surface.

\begin{prop}\label{prop3}
Let $M$ be a connected surface in $\R3$ without umbilic points. Then\vspace*{-\topsep}
\begin{enumerate}\itemsep0pt\def\labelenumi{$\alph{enumi})$}
\item $M$ is a \CRPC--surface if and only if $\tau_g$ and $\kappa_n$ are linearly dependent functions along the  isogonal lines.
\item $M$ is a \CSkC--surface if and only if $\tau_g$ is constant along the isogonal lines.
\end{enumerate}
\end{prop}
\begin{proof}
$a$) From Proposition \ref{p1.1}, we only need to show the converse part. If $\gamma$ is an  isogonal line (not a line of curvature) such that $c\tau_g+d\kappa_n=0$, for certain constants $c$ and $d$ with $c^2+d^2\neq0$, then $\gamma$ satisfies condition \emph{C1} of Proposition \ref{p1.1}. On the other hand, if two isogonal lines $\g_1$ and $\g_2$ in $M$ (not lines of curvature) intersect at a point, then condition \emph{C1} for $\g_1$ is equal to the condition \emph{C1} for $\g_2$. Using the isogonal flow and a standard continuity argument, the same condition \emph{C1} is valid over the whole surface, and so $M$ is an \emph{\CRPC--surface}.\par
\noindent $b$) This can be proved in a similar way to $a$).
\end{proof}

\section{Pseudo-geodesic lines}
\label{s:CCNS}\label{s:PGL}

Given a curve $\g$ in a surface $M$, the following facts are well known: 1) if $\g$ is geodesic and line of curvature, then $\g$ is a plane curve; 2) if $\g$ is geodesic and plane, then $\g$ is a line of curvature, \cite[pp. 211--212]{Spivak3}. However, if $\g$ is plane and line of curvature, then $\g$ is not necessarily a geodesic.
The following result, that can be seen as an extension of Lemma 15 in \cite[p. 212]{Spivak3}, is an easy consequence from the equation (\ref{kgkn}) and the relation (see \cite[p. 153]{DoCarmo76}):
\begin{equation}\label{taug}
\tau_\gamma =\tau_g+\theta'.
\end{equation}
\begin{prop}\label{equiv2}
Let $\gamma$ be an arc-length parametrized curve in $M$. If any two of the following conditions are satisfied then the third condition is also satisified:\vspace*{-\topsep}
\begin{enumerate}\itemsep-2pt\def\labelenumi{$\alph{enumi})$}
\item $\gamma$ is a plane curve.
\item $\gamma$ is a line of curvature.
\item $\gamma$ is a pseudo-geodesic line.
\end{enumerate}
\end{prop}

In the following we present several results connecting the concepts of isogonal line, pseudo-geodesic line and generalized helix.

\begin{prop}\label{p5.0}
Let $\g$ be a pseudo-geodesic line in a surface $M$, and assume that it is not an asymptotic line. Then $\g$ is a generalized helix if and only if $\kappa_n$ and $\tau_g$ are linearly dependent functions along $\g$.
\end{prop}
\begin{proof}
Note that on a pseudo-geodesic line $\g$ ($\theta$ is constant) we have from (\ref{taug}) and (\ref{kgkn}) that $\tau_g=\tau_\g$ and $\kappa_n=\cos\theta\kappa_\g$. Therefore, $\kappa_\g$ and $\tau_\g$ are linearly dependent functions along $\g$ (i.e. $\g$ is a generalized helix) if and only if $\kappa_n$ and $\tau_g$ are linearly dependent functions along $\g$.
\end{proof}

\begin{prop}\label{p5.1}
Let $\g$ be an isogonal line in a surface $M$ which also is a pseudo-geodesic line, and assume that it is not a line of curvature nor an asymptotic line. Then $\g$ is a generalized helix if and only if $\kappa_1$ and $\kappa_2$ are linearly dependent functions along $\g$.
\end{prop}
\begin{proof}
If $\g$ is a generalized helix then $m\kappa_\g+n\tau_\g=0$ for certain constants $m$ and $n$, and so $m\sec\theta \kappa_n+n\tau_g=0$. Then from Proposition \ref{p1.1} we get that $\kappa_1$ and $\kappa_2$ are linearly dependent functions along $\g$.

To proof the converse, let us assume that $\kappa_1$ and $\kappa_2$ are linearly dependent functions along $\g$. From Proposition \ref{p1.1} we get that $\kappa_n$ and $\tau_g$ are linearly dependent functions along $\g$, and so from Proposition \ref{p5.0} we deduce $\g$ is a generalized helix.
\end{proof}

\begin{prop}\label{p5.2}
Let $\g$ be a pseudo-geodesic line in a surface $M$, $\g$ without umbilical points. If $\g$ is a generalized helix and the curvatures $\kappa_1$ and $\kappa_2$ are linearly dependent along $\g$, then $\g$ is an isogonal line in $M$.
\end{prop}
\begin{proof}
If $\g$ is not an asymptotic line, then the result follows from Proposition \ref{p1.2}. In the case where $\g$ is an asymptotic line, then along $\g$ we have
\[
\cos^2\phi\ \kappa_1+\sin^2\phi\ \kappa_2=0,\quad a\kappa_1+b\kappa_2=0,
\]
for certain constants $a$ and $b$, with $a^2+b^2\neq0$. Then $\phi$ is necessarily a constant.
\end{proof}

\begin{prop}\label{hg1}
Let $\gamma$ be a generalized helix, with axis $V_\gamma$, in a surface $M$. Then
$\gamma$ is a pseudo-geodesic line in $M$ if and only if $\<V_\gamma,N|_\gamma\>$ is constant.
\end{prop}
\begin{proof}
If $\gamma$ is a generalized helix with axis $V_\gamma=\cos\psi\,T_\gamma+\sin\psi\,B_\gamma$, then the quotient $\tau_\gamma/\kappa_\gamma=-\cot\psi$ is constant. On the other hand, from (\ref{refdarboux}) we have $\<V_\gamma,N|_\gamma\>=\sin\psi\,\sin\theta$, which leads to the result.
\end{proof}
\bigskip

Let $M$ and $\ol{M}$ be two regular surfaces that intersect transversally (i.e. whenever $p\in M\cap\ol{M}$ we have $T_pM\neq T_p\ol{M}$), then it is well known that $C=M\cap\ol{M}$ is a regular curve. Let $\g(s)$ be an arc-length parametrization of $C$ and denote by $\xi(s)\in(0,\pi)$ the angle between $N(\g(s))$ and $\ol{N}(\g(s))$, where, as usual, a bar over a function or vector field denotes that we are considering the curve within the surface $\ol M$. A classical result states that if $M$ and $\ol{M}$ intersect with constant angle along a line of curvature of $M$, then the curve of intersection is also a line of curvature of $\ol{M}$. In 1846, Joachimsthal proved a special case of this theorem, \cite{Joachimsthal}, while the general case was presented by Bonnet in 1853, \cite{Bonnet}.

It is easy to see that
\begin{equation}\label{*}
\xi(s)=\e(\bar\theta(s)-\theta(s)),\quad \e\in\{\pm1\},
\end{equation}
where $\theta(s)$ (resp. $\bar\theta(s)$) is the angle between the vectors $N_\gamma(s)$ and $N(\gamma(s))$ (resp. $\ol{N}(\gamma(s))$). From here, we have $\xi'(s)=\e(\tau_g(s)-\ol{\tau}_g(s))$.
If we assume that $\gamma$ is of constant geodesic torsion $\tau_g$ in $M$, then $\gamma$ is of constant geodesic torsion $\ol{\tau}_g$ in $\ol M$ (with $\tau_g=\ol{\tau}_g$) if and only if $M$ and $\ol{M}$ intersect with constant angle.

We finish with a result of same kind as Joachimsthal theorem, which is a direct consequence of (\ref{*}).

\begin{prop}\label{JoachimsthalEx2}
Let $\gamma(s)$ be an arc-length parame\-tri\-zation of the intersection curve of two surfaces $M$ and $\ol M$ that intersect transversally, and assume that $\gamma$ is a pseudo-geodesic line in $M$. Then $\gamma$ is a pseudo-geodesic line in $\ol M$ if and only if $M$ and $\ol{M}$ intersect with constant angle.
\end{prop}

\subsection{Differential equations of pseudo-geodesic lines}

In this subsection, the angle $\theta$ between $N$ and $N_\g$ is not necessarily a constant. Let $X(t,z)$ be an orthogonal parametrization in a surface $M$, and denote by $\{E,F=0,G\}$ and $\{e,f,g\}$ the coefficients of the first and second fundamental forms, respectively.
Let $\g(s)=X(t(s),z(s))$ be an arc-length parametrized curve, then its unit tangent $T_\g=t'X_t+z'X_z$.
By taking derivative here we deduce the following ODE system
\begin{align*}
 t''+\Gamma_{11}^1\,(t')^2+2\Gamma_{12}^1\, t'z'+\Gamma_{22}^1\,(z')^2 = & -\sin\theta\,z'\sqrt{G/E}\,\kappa_\g, \\
 z''+\Gamma_{11}^2\,(t')^2+2\Gamma_{12}^2\, t'z'+\Gamma_{22}^2\,(z')^2 = & \sin\theta\,t'\sqrt{E/G}\,\kappa_\g,\\
 e(t')^2+2ft'z'+g(z')^2= & \cos\theta\,\kappa_\g,
\end{align*}
where as usual the $\Gamma_{ij}^k$'s denote the Christoffel symbols. If $\theta=\theta(s)\neq\pm\pi/2$, then the above equations lead to
\begin{align}
 t''+&(t')^2\,(\Gamma_{11}^1+\tan\theta\,z'\sqrt{G/E}\,e)+2 t'z'\,(\Gamma_{12}^1+\tan\theta\,z'\sqrt{G/E}\,f)+\nonumber\\
 &+(z')^2\,(\Gamma_{22}^1+\tan\theta\,z'\sqrt{G/E}\,g) =  0,\label{ec th cpnc1} \\
 z''+&(t')^2\,(\Gamma_{11}^2-\tan\theta\,t'\sqrt{E/G}\,e)+2 t'z'\,(\Gamma_{12}^2-\tan\theta\,t'\sqrt{E/G}\,f)+\nonumber\\
 &+(z')^2\,(\Gamma_{22}^2-\tan\theta\,t'\sqrt{E/G}\,g) = 0.\label{ec th cpnc2}
\end{align}

Conversely, if $\gamma(s)=X(t(s),z(s))$ satisfies equations (\ref{ec th cpnc1}) and (\ref{ec th cpnc2}) for a function $\bar{\theta}$, define the function $\bar{\kappa}=\sec\bar{\theta}\,(e(t')^2+2ft'z'+g(z')^2)$, and then we have
\begin{align*}
 t''+\Gamma_{11}^1\,(t')^2+2\Gamma_{12}^1\, t'z'+\Gamma_{22}^1\,(z')^2 = & -\sin\bar\theta\,z'\sqrt{G/E}\,\bar{\kappa},\\
 z''+\Gamma_{11}^2\,(t')^2+2\Gamma_{12}^2\, t'z'+\Gamma_{22}^2\,(z')^2 = & \sin\bar\theta\,t'\sqrt{E/G}\,\bar{\kappa}.
\end{align*}
By derivating $T_\g$ we get $\kappa_\gamma\,N_\gamma=t''X_t+z''X_z+t'^2X_{tt}+2t'z'X_{tz}+z'^2X_{zz}$, and then we deduce
\begin{align*}
  \kappa_\gamma\<N_\gamma,X_t\>= & (t''+\Gamma_{11}^1\,(t')^2+2\Gamma_{12}^1\, t'z'+\Gamma_{22}^1\,(z')^2)\,E, \\
  \kappa_\gamma\<N_\gamma,X_z\>=  & (z''+\Gamma_{11}^2\,(t')^2+2\Gamma_{12}^2\, t'z'+\Gamma_{22}^2\,(z')^2)\,G, \\
  \kappa_\gamma\<N_\gamma,N\>=  & e(t')^2+2ft'z'+g(z')^2.
\end{align*}
From these three equations, and bearing (\ref{refdarboux}) in mind, we get
\begin{equation*}
  \kappa_\gamma\,\sin{\theta} =\bar{\kappa}\,\sin\bar\theta, \quad\text{and}\quad  \kappa_\gamma\,\cos{\theta} =\bar{\kappa}\,\cos\bar\theta,
\end{equation*}
from which we deduce $\bar{\theta}=\theta$ and $\bar{\kappa}=\kappa_\gamma$.

In conclusion, we have shown the following result (note that when $\theta=0$ we recover the equations of geodesic curves).

\begin{prop}
Let $X(t,z)$ be an orthogonal parametrization in a surface $M$. A curve $\g(s)=X(t(s),z(s))$ is a pseudo-geodesic line if and only if equations (\ref{ec th cpnc1}) and (\ref{ec th cpnc2}) are satisfied.
\end{prop}

\section{The main result}
\label{s:MR}

The goal of this section is to show that helix surfaces and the Enneper surface are the only ones whose isogonal lines are generalized helices and pseudo-geodesic lines. First, we will see that helix surfaces and the Enneper surface verify that property.

Before starting with the examples, we recall the well-known formula of Liouville (see \cite[p. 253]{DoCarmo76}):
\begin{equation}
\kappa_g = \phi'+\cos\phi\,\kappa_{g_1}+\sin\phi\,\kappa_{g_2}, \label{kg}
\end{equation}
where $\kappa_{g_1}$ and $\kappa_{g_2}$ denote the geodesic curvatures of the lines of curvature.

\begin{example}\label{ex1}
A surface $M$ in $\R3$ is a \emph{helix surface} if there exists an unitary vector $V$ such that for each point $q\in M$ the angle between $V$ and $T_qM$ is constant, \cite{DSRH09}. A helix surface can be locally parametrized as follows (\cite{DSRH09,LO16a}),
\begin{equation}\label{ParamtSupHelice}
X(t,z)=\beta(t)+z\left(\cos\varphi\,N_\beta(t)+\sin\varphi\, V\right),
\end{equation}
where $\varphi$ is a constant angle and $\beta$ a curve in a plane orthogonal to $V$. The geodesics of helix surfaces are slant helices, and they are characterized by the fact that their unit normal vector field makes a constant angle with the vector $V$. The coordinate curves of $X(t,z)$ are lines of curvature, and we have the following formulas
\[
\kappa_{g_1}=\frac{\cos\varphi\;\kappa_\beta}{1-z\cos\varphi\;\kappa_\beta},\qquad \kappa_{g_2}=0, \qquad\kappa_{1}=\frac{-\sin\varphi\;\kappa_\beta}{1-z\cos\varphi\;\kappa_\beta},\qquad \kappa_{2}=0.
\]
If $\gamma$ is an isogonal line, from (\ref{kgkn}), (\ref{kntg}), (\ref{kg}) and Proposition \ref{p5.1} we easily get that $\g$ is a generalized helix and a pseudo-geodesic line.

Moreover, cylinders are the only helix surfaces whose  isogonal lines are generalized helices and geodesics.
\hfill$\Box$
\end{example}

\begin{example}\label{ex2}
The usual parametrization of the \emph{Enneper surface}, \cite[p. 205]{DoCarmo76}, whose coordinate curves are lines of curvature, is given by
\begin{equation}\label{ParamtEnneper}
X(t,z)=\Big(t-\frac{t^3}{3}+tz^2,z-\frac{z^3}{3}+zt^2,t^2-z^2\Big).
\end{equation}
A straightforward computation yields
\[
\kappa_{g_1}=\frac{-2 z}{(1+t^2+z^2)^2},\quad \kappa_{g_2}=\frac{2t}{(1+t^2+z^2)^2},\quad \kappa_{1}=\frac{2}{(1+t^2+z^2)^2},\quad \kappa_{2}=\frac{-2}{(1+t^2+z^2)^2}.
\]
A curve $\gamma(s)=X(t(s),z(s))$ is an isogonal line if and only if $t'(s)\sqrt{E}=\cos\phi$ and $z'(s)\sqrt{G}= \sin\phi$, for a constant $\phi$, or equivalently $z'(s)=\tan\phi\,t'(s)$. Hence, the  isogonal lines are given by  $\gamma(s)=X(t(s),m\,t(s)+n)$, for certain constants $m=\tan\phi$ and $n$. In this case, from (\ref{kgkn}), (\ref{kntg}) and (\ref{kg}) 
we get
\[
\tan\theta= \frac{-n\cos\phi}{\cos^2\phi-\sin^2\phi},
\]
showing that $\gamma$ is a pseudo-geodesic line (without loss of generality we can assume that $\phi\neq\pi/4$; otherwise, $\g$ is a line of curvature). Furthermore, from Proposition \ref{p5.1} we also obtain that $\g(s)$ is a generalized helix.

In summary, if $\gamma(s)$ is an isogonal line in the Enneper surface, then $\g$ satisfies the following properties: (i) $\gamma$ is a pseudo-geodesic line; (ii) $\g$ is a generalized helix; (iii) $\g(s)=X(\alpha(s))$, where $\alpha$ is a straight line in the plane.
\hfill$\Box$
\end{example}

In the following two examples we show surfaces which do not verify the above properties, but which will be used in the proof of the theorem.

\begin{example}\label{ex3}
It is shown in \cite[p. 96]{Kuhnel06} that a revolution surface which is a \CRPC--surface can be parametrized as
\[
X(t,z)=(t\,\cos z,t\,\sin z,\varepsilon \int_1^t u^c(1-u^{2c})^{-1/2}du),
\]
where $\varepsilon=\pm 1$ and $\kappa_1=c\kappa_2\neq 0$.
The principal directions and the unit normal are given by
\begin{align*}
E_1 &= (\sqrt{1-t^{2c}}\,\cos z,\sqrt{1-t^{2c}}\,\sin z,\varepsilon t^c)\\
E_2 &= (-\sin z,\cos z,0),\\
N &=(-\varepsilon t^c\,\cos z,-\varepsilon t^c\,\sin z,\sqrt{1-t^{2c}}).
\end{align*}
Let $\gamma(s)=X(t(s),z(s))$ be an  isogonal line which is not a line of curvature. Then we have
\begin{equation}\label{cpcSR}
t'= \cos\phi\sqrt{1-t^{2c}} \qquad\text{and}\qquad  z'\,t= \sin\phi,
\end{equation}
and therefore
\[
T_\gamma=(\cos\phi\,\sqrt{1-t^{2c}}\,\cos z-\sin\phi\,\sin z,\cos\phi\,\sqrt{1-t^{2c}}\,\sin z+\sin\phi\,\cos z,\varepsilon \cos\phi\,t^c).
\]
From here and (\ref{cpcSR}) we get
\begin{equation}\label{cpcSRcurv}
\kappa^2_\gamma=\frac{1}{t^2}\left(t^{2c}(c^2\cos^4\phi+(2c-1)\cos^2\phi\sin^2\phi)+\sin^2\phi\right).
\end{equation}
If $\gamma$ were a pseudo-geodesic line, then there would exist a constant $\theta$ such that
\begin{equation}\label{cpcSRcurv1}
\kappa_\gamma\cos\theta=\<\kappa_\gamma N_\gamma,N\>=\e(\sin^2\phi+c\cos^2\phi)t^{c-1}.
\end{equation}
Then from (\ref{cpcSRcurv}) and (\ref{cpcSRcurv1}) we deduce that either $\g$ is a line of curvature or $t(s)$ is a constant function, which is a contradiction. In conclusion, we have proved that on this surface there exist isogonal lines which are not pseudo-geodesic lines.\hfill$\Box$
\end{example}

\begin{example}\label{ex4}
According to \cite[pp. 164--166]{Nitsche89} (see also \cite[p. 308]{Eis09}), the catenoid, the Enneper surface and the Bonnet surfaces are the only nonplanar minimal surfaces with plane lines of curvature. The usual parametrization of a Bonnet surface, whose coordinate curves are lines of curvature, is given by
\begin{equation}\label{supbonnet}
  X(t,z)=\left(\frac{1}{\sqrt{1-a^2}}(a t+\sin t\cosh z),\frac{1}{\sqrt{1-a^2}}(z+ a\cos t\sinh z), \cos t\cosh z\right),
\end{equation}
for $0<a<1$. It is not difficult to see that
\[
\kappa_{g_1}=\frac{-\sqrt{1-a^2}\sinh z}{(a\cos t+\cosh z)^2},\quad\kappa_{g_2}=\frac{-a\sqrt{1-a^2}\sin t}{(a\cos t+\cosh z)^2},\quad\kappa_{2}=-\kappa_1=\frac{1-a^2}{(a\cos t+\cosh z)^2}.
\]

Let us suppose that $\gamma$ is an  isogonal line and a pseudo-geodesic line. From (\ref{kgkn}), (\ref{kntg}) and (\ref{kg}), and reasoning as in Example \ref{ex2}, we can write
\[
A\,\sinh z+B\,\sin t+C=0,
\]
for certain constants $A$, $B$ and $C$. Derivating twice here, and taking into account that $z'=\tan\phi\; t'$, we get a contradiction. In conclusion, as in the above example, we have proved that on this surface there exist  isogonal lines which are not pseudo-geodesic lines.
\hfill$\Box$
\end{example}

Now we can state and prove the main result.

\begin{theo}\label{teoprinc}\label{MR}
Let $M$ be a nonplanar connected surface in $\R3$. Then the  isogonal lines are generalized helices and pseudo-geodesic lines if and only if $M$ is a helix surface or an open piece of the Enneper surface.
\end{theo}
\begin{proof}
Bearing examples \ref{ex1} and \ref{ex2} in mind, we only need to show the direct part.

Let $M$ be a connected surface such that any  isogonal line is a generalized helix and a pseudo-geodesic line. From Proposition \ref{equiv2}, $M$ has plane lines of curvature. On the other hand, if $\g$ is a pseudo-geodesic line and a generalized helix, then $\kappa_n$ and $\tau_g$ are linearly dependent functions along $\g$, and from Proposition \ref{prop3} we deduce that $M$ is a \emph{\CRPC}-surface.

In \cite{Kim10}, the authors study surfaces with $k_2=ak_1+b$ and plane lines of curvature, and show that (i) if $M$ is a flat surface then $M$ is a helix surface (also called a slant cylinder); (ii) if $M$ is nonflat then $M$ is a minimal surface or a surface of revolution. In case (ii), and bearing examples \ref{ex3} and \ref{ex4} in mind, we get that $M$ is an open piece of the Enneper surface. This concludes the proof.
\end{proof}

\section{Appendix: A special family of geodesics in the Enneper surface}
\label{s:App}

Another interesting property of the Enneper surface $E$ is the following: it is path-connected by generalized helices, that is, given two points $p_1$ and $p_2$ in $E$ there exists a generalized helix in $E$ connecting both points.
Indeed, let $p_1=X(t_1,z_1)$ and $p_2=X(t_2,z_2)$, with $t_1\neq t_2$, two points in $E$ and consider the curve $\gamma(t)=X(t,m\,t+n)$, with $m=(z_2-z_1)/(t_2-t_1)$ and $n=(t_2z_1-t_1z_2)/(t_2-t_1)$. Then $\g$ is an isogonal line in $E$ and so it is a generalized helix. In the case $t_1=t_2$, the points $p_1$ and $p_2$ are connected by a planar line of curvature. Some generalized helices are plotted in figure \ref{fig.ex1}.

\begin{figure}\centering
\includegraphics[height=8cm]{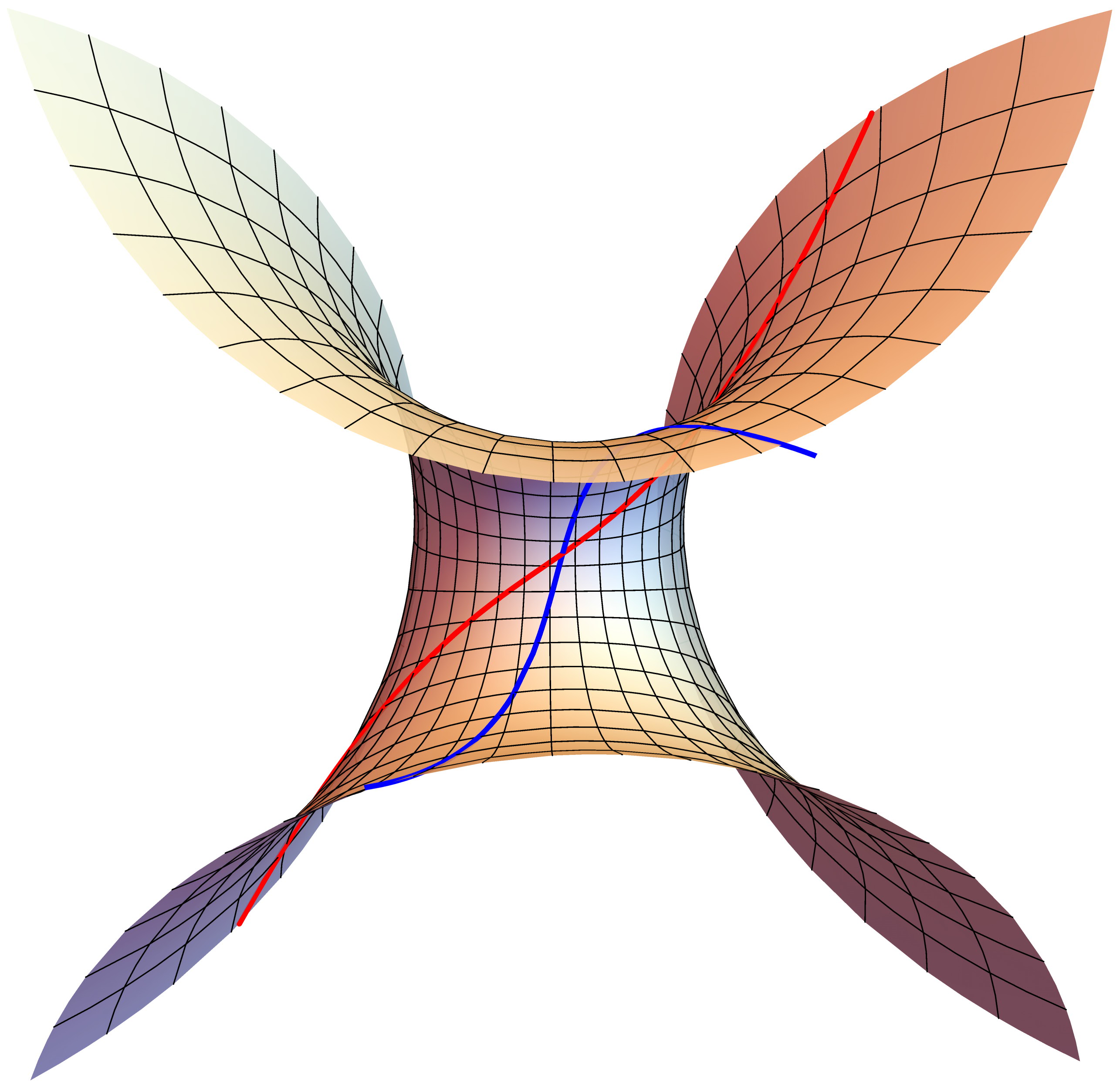}
\caption{\label{fig.ex1}The Enneper surface and two generalized helices passing through the origin.}
\end{figure}

It is easy to see that
\begin{align}\label{tanghgennep}
T_\gamma(t)= & \frac{\cos\phi}{1+t^2+(mt+n)^2}\Big(1+n^2+4mnt+(3m^2-1)t^2, \nonumber\\
   & m(1-n^2)+2n(1-m^2)t+m(3-m^2)t^2,-2mn+2(1-m^2)t\Big),
\end{align}
and the axis of $\g$ is given by
\[
 W=\frac{1}{\sqrt{1+m^2+n^2}}(m,1,-n).
\]
From (\ref{refdarboux}) we get
\[
\sin\psi\sin\theta=\<W,N|_\gamma\>=\frac{n}{\sqrt{1+m^2+n^2}},
\]
where $\tau_\g/\kappa_\g=-\cot\psi$, and so the  isogonal lines passing through the origin are geodesics. Hence, we have shown that the family of geodesics of the Enneper surface passing through the origin consists of the  isogonal lines passing through the origin, which are also generalized helices. This family can be described as
\begin{equation}\label{familiageodesicas}
\gamma_m(t)=\frac{1}{3}\Big(3t+(3m^2-1) t^3,m(3t-(m^2-3)t^3),3(1-m^2)t^2\Big),
\end{equation}
where $m=\tan\phi$.

\section*{Acknowledgements}

This research is part of the grant PID2021-124157NB-I00, funded by MCIN/ AEI/ 10.13039/ 501100011033/ ``ERDF A way of making Europe''. Also supported by ``Ayudas a proyectos para el desarrollo de investigación científica y técnica por grupos competitivos'', included in the ``Programa Regional de Fomento de la Investigación Científica y Técnica (Plan de Actuación 2022)'' of the Fundación Séneca-Agencia de Ciencia y Tecnología de la Región de Murcia, Ref. 21899/PI/22.

\end{document}